\documentclass[12pt,reqno]{amsart}

\usepackage{amsmath,amsfonts,amsthm,amssymb,amsxtra}
\usepackage{bbm} 
\usepackage{bm} 

\newtheorem{theorem}{Theorem}[section]
\newtheorem{proposition}[theorem]{Proposition}
\newtheorem{lemma}[theorem]{Lemma}

\theoremstyle{definition}

\newtheorem{definition}[theorem]{Definition}

\theoremstyle{remark}

\numberwithin{equation}{section}

\newcommand{\const}{\mathrm{const}\ }

\renewcommand{\epsilon}{\varepsilon}

\newcommand{\N}{\mathbb{N}}

\renewcommand{\phi}{\varphi}
\newcommand{\R}{\mathbb{R}}

\newcommand{\Z}{\mathbb{Z}}

\DeclareMathOperator{\supp}{supp}

\def\br{\mathbb{R}}
\def\bn{\mathbb{N}}
\def\bz{\mathbb{Z}}

\def\cl{\mathcal{L}}

\def\td{\mathrm{d}}
\def\dk{\mathrm{d}k}
\def\ds{\mathrm{d}s}
\def\dt{\mathrm{d}t}
\def\dx{\mathrm{d}x}
\def\dy{\mathrm{d}y}
\def\dz{\mathrm{d}z}

\newcommand{\abs}[1]{\left\lvert#1\right\rvert}
\newcommand{\norm}[1]{\left\lVert#1\right\rVert}
\newcommand{\me}[1]{\mathrm{e}^{#1}}
\newcommand{\one}{\mathbf{1}}

\makeatletter
\newcommand*{\rom}[1]{\expandafter\@slowromancap\romannumeral #1@}
\makeatother

\begin{document}

\title[Generalized Hardy operators on $L^p$]{On Scales of Sobolev spaces\\
  associated to generalized Hardy operators}

\author{Konstantin Merz}
\address[Konstantin Merz]{Mathematisches Institut, Ludwig-Maximilans Universit\"at M\"unchen, Theresienstr. 39, 80333 M\"unchen, Germany. Address as of October 2019: Institut f\"ur Analysis und Algebra, Carolo-Wilhelmina, Universit\"atsplatz 2, 38106 Braunschweig, Germany}
\email{k.merz@tu-bs.de}


\date{November 4, 2020}

\begin{abstract}
  We consider the fractional Laplacian with Hardy potential and study the
  scale of homogeneous $L^p$ Sobolev spaces generated by this operator.
  Besides generalized and reversed Hardy inequalities, the analysis relies
  on a H\"ormander multiplier theorem which is crucial to construct a basic
  Littlewood--Paley theory.
  The results extend those obtained recently in $L^2$ but do not cover
  negative coupling constants in general due to the slow decay of the
  associated heat kernel.
\end{abstract}
\thanks{Deutsche Forschungsgemeinschaft grant SI 348/15-1 is gratefully
  acknowledged.}

\maketitle
\section{Introduction \& result}

\subsection*{Introduction}
The classical, sharp Hardy inequality
\begin{align*}
  \int_{\R^d}|\nabla f|^p\,\dx - \left(\frac{|d-p|}{p}\right)^p \int_{\R^d}\frac{|f(x)|^p}{|x|^p}\,\dx\geq0
\end{align*}
is one of the longest known inequalities relating the weighted $L^p$ norm of
a decaying function with the $L^p$ norm of its gradient and plays an
important role in fields such as mathematical physics, non-linear PDEs,
and harmonic analysis.
This inequality holds for all $f\in C_c^\infty(\R^d)$ if $1\leq p<d$ and
for all $f\in C_c^\infty(\R^d\setminus\{0\})$ if $p>d$. 
Herbst considered a generalization of the above inequality related to
the fractional Laplacian $(-\Delta)^{\alpha/2}$. Here, and in the following we
restrict ourselves to $\alpha\in(0,2)$. For $1<p<2d/\alpha$ the inequality
states
\begin{align}
  \label{eq:standardhardy}
  \| (-\Delta)^{\frac{\alpha}{4}}f\|_{L^p(\R^d)} - 2^{\frac{\alpha}{2}}\frac{\Gamma\left(\frac{d/p'+\alpha/2}{2}\right)\Gamma\left(\frac{d}{2p}\right)}{\Gamma\left(\frac{d/p-\alpha/2}{2}\right)\Gamma\left(\frac{d}{2p'}\right)}\| |x|^{-\frac{\alpha}{2}} f\|_{L^p(\R^d)} \geq0
  \quad\text{for all}\ f\in C_c^\infty(\R^d)
\end{align}
where the constant on the right side is sharp, see
\cite[Theorem 2.5]{Herbst1977}.
We emphasize that, for $p\neq2$, $\| (-\Delta)^{\alpha/4}f\|_p$ is
\emph{not} proportional to
$$
\left(\int_{\R^d}\int_{\R^d}\frac{|f(x)-f(y)|^p}{|x-y|^{d+\alpha p/2}}\,\dx\,\dy\right)^{1/p}\,.
$$
Instead, there is a one-sided inequality depending on whether $1<p<2$
or $p>2$, see, e.g., \cite[Chapter \rom{5}]{Stein1970} and also
Frank and Seiringer \cite{FrankSeiringer2008} concerning the sharp
fractional Hardy inequality involving this expression.
If
\begin{align*}
  -a_*:=\frac{2^\alpha\Gamma((d+\alpha)/4)^2}{\Gamma((d-\alpha)/4)^2}
\end{align*}
denotes the squared sharp constant for $p=2$, we define the
\emph{generalized Hardy operator}
$$
\cl_{a,\alpha}:=(-\Delta)^{\alpha/2}+a|x|^{-\alpha} \quad \text{in}\ L^2(\R^d)
$$
as the Friedrichs extension of the corresponding quadratic form on
$C_c^\infty(\R^d)$ for $a\geq a_*$.
For $d=3$ and $\alpha=1$, the optimal $a_*$ was already known to Kato
\cite[Chapter 5, Equation (5.33)]{Kato1966}. For general $d$ and $\alpha$
it was first computed by Herbst, but see also
\cite{Kovalenkoetal1981,Yafaev1999,Franketal2008H,FrankSeiringer2008} for
alternative proofs of the inequality with sharp constant.

Due to the homogeneity of $\cl_{a,\alpha}$ it is natural to ask whether the
operators with $a=0$ and $a\neq0$ are in some sense equivalent to each other.
For instance, one may ask whether they generate scales of homogeneous
Sobolev spaces which are comparable with each other, i.e., whether there
are $0<A<A'$ such that
$$
\|(-\Delta)^{\alpha s/4}f\|_{L^p(\R^d)} \leq A \|\cl_{a,\alpha}^{s/2}f\|_{L^p(\R^d)} \leq A'\|(-\Delta)^{\alpha s/4}f\|_{L^p(\R^d)}
$$
holds for certain $\alpha,a,s,p$.
For the Schr\"odinger operator, i.e., $\alpha=2$ and $d\geq3$, Killip et al.
\cite{Killipetal2018} proved that the norms are in fact equivalent to each
other for certain $a,s,p$.
This finding was recently generalized by Frank et al. \cite{Franketal2019} in
the case $p=2$ with general $\alpha\in(0,\min\{2,d\})$ and $a\geq a_*$.
The main objective of this paper is a generalization of their result to
$L^p(\R^d)$ with $p\neq2$.

Our results may be useful to study non-linear PDEs involving $\cl_{a,\alpha}$
in order to reduce problems to those involving only $|p|^\alpha$, i.e.,
without the Hardy potential. For $\alpha=2$, the corresponding result was
used, e.g., by Killip et al. \cite{Killipetal2017T,Killipetal2017} to
determine the threshold between scattering and finite-time blowup of the
focusing cubic nonlinear Schr\"odinger equation with Hardy
potential, or the well-posedness of the energy-critical nonlinear
Schr\"odinger equation with Hardy potential.

\medskip
Before proceeding to the main result, we introduce some
notation that is used throughout the rest of this paper.\\
(1) We write $X\lesssim Y$ for non-negative quantities $X$ and $Y$, whenever
there is a positive constant $A$ such that $X\leq A\cdot Y$.
If $A$ depends on some parameter $r$, we sometimes write $X\lesssim_r Y$.
Moreover, $X\sim Y$ means $Y\lesssim X\lesssim Y$ and in this case, we say
that $X$ is \emph{equivalent} to $Y$.\\
(2) We define $X\wedge Y:=\min\{X,Y\}$ and $X\vee Y:=\max\{X,Y\}$.\\
(3) The integer part of $x\in\R$ is denoted by $[x]:=\max\{k\in\Z:k\leq x\}$.
The positive part is denoted by $x_+=\max\{x,0\}$.\\
(4) For $1\leq p\leq\infty$, we abbreviate $\|f\|_p=\|f\|_{L^p(\R^d)}$ and
denote H\"older conjugate exponents by primes, i.e., $p^{-1}+p'^{-1}=1$.
\\ 
(5) The $s$-th $L^2$ potential space ($s\in\R$) is denoted by $H^{s}(\R^d)$.
It consists of all functions $f$ such that the norm
$\|f\|_{H^s}:=\|(1-\Delta)^{s/2}f\|_{L^2}$ is finite where $(1-\Delta)^{s/2}$
denotes the operator which is given by multiplication with
$(1+4\pi^2|\xi|^2)^{s/2}$ in Fourier space.
Moreover, $f\in H^s_{\mathrm{loc}}(\R^d)$ if and only if
$\|\phi f\|_{H^s}<\infty$ for all $\phi\in C_c^\infty(\R^d)$.
\\
(6) We abbreviate $|p| = \sqrt{-\Delta}$.

\subsection*{Main result and strategy of the proof}

Let us recall the following parameterization of the coupling constant
in terms of the power of the formal ground state of $\cl_{a,\alpha}$,
namely
\begin{align}
  \label{eq:defpsi}
  \Psi_{\alpha,d}(\sigma) := -2^\alpha \frac{\Gamma(\frac{\sigma+\alpha}{2})\ \Gamma(\frac{d-\sigma}{2})}{\Gamma(\frac{d-\sigma-\alpha}{2})\ \Gamma(\frac{\sigma}{2})}
  \qquad\text{if}\ \sigma\in (-\alpha,(d-\alpha)/2]\setminus\{0\}
\end{align}
and $\Psi_{\alpha,d}(0)=0$. According to \cite[Lemma 3.2]{Franketal2008H} and
\cite[p. 8]{JakubowskiWang2020}, the function
$\sigma\mapsto \Psi_{\alpha,d}(\sigma)$ is continuous and strictly decreasing
in $(-\alpha,(d-\alpha)/2]$ with
$$
\lim_{\sigma\to -\alpha} \Psi_{\alpha,d}(\sigma) = \infty
\qquad\text{and}\qquad
\Psi_{\alpha,d}\left(\frac{d-\alpha}{2}\right) = a_* \,.
$$
Consequently, for any $a\geq a_*$, we may define
\begin{equation}
  \label{eq:defdelta}
  \delta := \Psi_{\alpha,d}^{-1}(a)
\end{equation}
which allows us to formulate our main theorem on the equivalence of
$L^p$ Sobolev norms generated by powers of $\cl_{a,\alpha}$.
\begin{theorem}[Equivalence of Sobolev norms on $L^p(\R^d)$]
  \label{Thm:A.1}
  Let $d\in\N$, $0<\alpha<2\wedge d$, and $s\in(0,2]$.
  Let $a\geq a_*$ if $s=2$ and $a\geq0$ if $s\in(0,2)$.
  Let furthermore $\delta$ be defined by \eqref{eq:defdelta}.
  \begin{enumerate}
  \item If $1<p<\infty$ satisfies $\alpha s/2+\delta<d/p<\min\{d,d-\delta\}$,
    then
    $$
    \| |p|^{\alpha s/2}f\|_{L^p(\R^d)} \lesssim_{d,p,\alpha,s} \|\cl_{a,\alpha}^{s/2}f\|_{L^p(\R^d)}
    \quad \text{for all}\ f\in C_c^\infty(\R^d)\,.
    $$
    
  \item
    If $\alpha s/2<d/p<d$ (which already ensures $1<p<\infty$), then
    $$
    \| \cl_{a,\alpha}^{s/2}f\|_{L^p(\R^d)} \lesssim_{d,p,\alpha,s} \| |p|^{\alpha s/2}f\|_{L^p(\R^d)}
    \quad \text{for all}\ f\in C_c^\infty(\br^d).
    $$
  \end{enumerate}
\end{theorem}


We remark that for $p=2$, an equivalence of Sobolev norms for some
$s\in(0,2]$ and $a\geq a_*$ yields, by the spectral theorem and the
operator monotonicity of positive roots, an equivalence of norms for
any $0<t<s$ with the same $a$, see also
\cite[Remarks 1.2 and 1.3]{Franketal2019}.
If $p\neq2$, this assertion is far from obvious.

Let us now outline the strategy of the proof.
For $s=2$, the assertion follows immediately from the ordinary Hardy
inequality \eqref{eq:standardhardy} and a \emph{generalized Hardy
  inequality} which is why we can also handle $a<0$ in this case.

\begin{proposition}[Generalized Hardy inequality]
  \label{genhardy}
  Let $1<p<\infty$, $\alpha\in(0,2\wedge d)$, $a\geq a_*$, $\delta$ be
  defined by \eqref{eq:defdelta}, and $\alpha s/2\in(0,d)$.
  If $s$ and $p$ satisfy $\alpha s/2+\delta<d/p<d-\delta$, then
  \begin{align}
    \label{eq:genhardy}
    \| |x|^{-\alpha s/2}f\|_p \lesssim_{d,\alpha,a,s,p} \|\cl_{a,\alpha}^{s/2}f\|_p 
    \quad \text{for all}\ f\in C_c^\infty(\br^d) \,.
  \end{align}
  Conversely, if $\alpha s/2\in(0,\min\{d,d-2\delta\})$ and the above
  estimate holds, then $\alpha s/2+\delta<d/p<d-\delta$.
\end{proposition}

\begin{proof}
  The assertion is equivalent to the $L^p$-boundedness of the
  operator $|x|^{-\alpha s/2}\cl_{a,\alpha}^{-s/2}$. Using the pointwise
  bounds on the Riesz kernel of $\cl_{a,\alpha}$ (see
  \cite[Theorem 1.6]{Franketal2019}), i.e.,
  \begin{align}
    \label{eq:rieszhardy}
    \cl_{a,\alpha}^{-s/2}(x,y) \sim_{d,\alpha,a,s} |x-y|^{\alpha\frac{s}{2}-d}\left(1\wedge\frac{|x|}{|x-y|}
    \wedge\frac{|y|}{|x-y|}\right)^{-\delta}
    \quad \text{for}\ \frac{\alpha s}{2}\in(0,d\wedge d-2\delta)\,,
  \end{align}
  this follows from the $L^p$-boundedness of the operator whose integral
  kernel is the above kernel multiplied by $|x|^{-\alpha s/2}$ which in turn
  is proven by performing a Schur test. As the involved computations are
  analogous those in \cite[Proposition 3.2]{Killipetal2018} (with $\alpha s/2$
  instead of $\alpha$ and $\delta$ instead of $\sigma$), respectively
  \cite[Proposition 1.4]{Franketal2019}, we omit a proof.

  The fact that \eqref{eq:genhardy} fails for $d/p\leq \alpha s/2 + \delta$
  or $d/p\geq d-\delta$ follows from the lower bound in \eqref{eq:rieszhardy}
  and the same counterexamples as in \cite[Proposition 3.2]{Killipetal2018}.
\end{proof}

\medskip

If $s<2$, the proof is a bit more laborious.
Still, the idea is to use the triangle inequality, obtain an estimate like
\begin{align}
  \label{eq:reversehardywish}
  \||p|^{\alpha s/2}f\|_p
  \leq \|(\cl_{a,\alpha}^{s/2}-|p|^{\alpha s/2})f\|_p + \| \cl_{a,\alpha}^{s/2} f\|_p
  \lesssim\||x|^{-\alpha s/2}f\|_p + \| \cl_{a,\alpha}^{s/2} f\|_p\,,
\end{align}
and then apply the generalized Hardy inequality.
For $p=2$, the second inequality in \eqref{eq:reversehardywish} was called a
\emph{reversed Hardy inequality} (cf. \cite[Proposition 1.5]{Franketal2019}),
because it yields a lower bound on the norm of $|x|^{-\alpha s/2}f$ in terms
of the difference $(\cl_{a,\alpha}^{s/2}-|p|^{\alpha s/2})f$.
There, \eqref{eq:reversehardywish} was proven using the spectral theorem, i.e.,
\begin{align*}
  \Gamma\left(\frac s2\right) \|\cl_{a,\alpha}^{s/2}f\|_2
  & = \left\|\int_0^\infty\frac{\dt}{t}t^{-\frac s2}(1-\me{-t\cl_{a,\alpha}})f\right\|_2\\
  & \leq \left\|\int_0^\infty\frac{\dt}{t}t^{-\frac s2}(\me{-t|p|^\alpha}-\me{-t\cl_{a,\alpha}})f\right\|_2
  + \left\|\int_0^\infty\frac{\dt}{t}t^{-\frac s2}(1-\me{-t|p|^\alpha})f\right\|_2\,,
\end{align*}
thereby rewriting the difference $\|(\cl_{a,\alpha}^{s/2}-|p|^{\alpha s/2})f\|_2$
directly in terms of the difference of the associated heat kernels.
However, due to the lack of a spectral theorem in $L^p$, we will first 
express $\|\cl_{a,\alpha}^{s/2}f\|_p$ and $\||p|^{\alpha s/2}f\|_p$
in terms of Littlewood--Paley square functions employing two-sided
\emph{square function estimates}, Theorem \ref{squarefunctions}.
The corresponding Littlewood--Paley projections will be defined via the
heat kernels of $\cl_{a,\alpha}$ and $|p|^\alpha$, because we have good
pointwise bounds on the individual kernels and on their difference,
Theorem \ref{heatkernel} and Lemma \ref{differencekernel}.
The latter bounds then allow us to prove a reversed Hardy inequality
expressed in terms of these square functions, thereby yielding an
analog of \eqref{eq:reversehardywish}, see Proposition \ref{reversehardylp}.

The proof of the square function estimates, however, crucially depends on
the $L^p$-boundedness of certain functions of $\cl_{a,\alpha}$.
In $L^2$ it follows from the spectral theorem that measurable, bounded
functions of self-adjoint operators are bounded on $L^2$.
The $L^p$-boundedness of functions of such operators (which may initially
be defined by the $L^2$ functional calculus), however, relies on much stronger
regularity assumptions on the multiplier and specific knowledge of the operator
in question.
Here, we discuss two instances of such spectral multiplier
theorems which differ in the conditions on the multiplier.
On the one hand, Mikhlin multiplier theorems \cite{Mihlin1956} require that the
multiplier $m$ is at least $s$ times continuously differentiable and
satisfies the Mikhlin condition
\begin{align*}
  |\lambda^{j}\partial_\lambda^j m(\lambda)| \lesssim_j 1 \quad \text{for all}\ j=0,...,s\,.
\end{align*}
On the other hand, H\"ormander multiplier theorems 
rely on the condition that the multiplier $F$ belongs to
$H^s_{\mathrm{loc}}(\R)$ for some sufficiently large $s>0$. Moreover,
for a non-zero $\phi\in C_c^\infty(\R_+)$, the H\"ormander condition
\begin{align*}
  \sup_{t>0}\|\varphi(\cdot) F(t\cdot)\|_{H^{s}}<\infty
\end{align*}
must be satisfied. This reveals in particular that H\"ormander
multiplier theorems imply Mikhlin multiplier theorems.
It is known that $s>d/2$ suffices to prove a Mikhlin or a H\"ormander
multiplier theorem for Fourier multipliers, see
\cite[Chapter \rom{4}, \S 3, Theorem 3]{Stein1970} and
\cite{Hormander1960}.

There is a broad literature on the derivation of spectral multiplier theorems.
However, these usually rely on the assumption that the corresponding heat
kernel satisfies pointwise Gaussian estimates
\cite{Hebisch1990,Hebisch1990A,Duongetal2002,Chenetal2016} or
so-called generalized Gaussian estimates \cite{Blunck2003}. The kernel may
even have local singularities, like the one of $-\Delta+a|x|^{-2}$ for $a<0$,
see, e.g., Milman and Semenov \cite{MilmanSemenov2004}.
For a survey on spectral multiplier theorems for operators with Gaussian
heat kernel bounds, we refer to Duong et al. \cite{Duongetal2001} and the
references therein.
Using the maximum principle and the exponential decay of $\exp(\Delta)$,
Hebisch \cite{Hebisch1990} derived a multiplier theorem for Schr\"odinger
operators $-\Delta+V$ in $L^2(\R^d)$ when $V\geq0$. Unlike in an earlier
work \cite{Hebisch1990A} where the heat kernel needed to satisfy a certain
H\"older condition, the proof relies on decent $L^2$ estimates and is
based on a clever dyadic decomposition of the multiplier.
Naturally, the maximum principle can also be invoked for
$\exp(-(|p|^\alpha +V))$ with $\alpha\in(0,2)$ and $V\geq0$. However, due
to the slow, i.e., algebraic decay of $\exp(-|p|^\alpha)$ it is considerably
more difficult to show a multiplier theorem also in this case.
Using similar techniques as in \cite{Hebisch1990}, it is however possible
to prove a H\"ormander multiplier theorem for $|p|^\alpha+V$, at least in the
special case $d=1$ and $\alpha>1$, see \cite[Theorem 3.8]{Hebisch1995}.
The reason for this restriction is the slow decay of the heat kernel which
makes it difficult to deduce radial, integrable upper bounds of functions of
$\cl_{a,\alpha}$, even if these functions are smooth and compactly supported.
The existence of such upper bounds is, however, vital to make use of a
well-known property of the Hardy--Littlewood maximal function in order to
conclude the proof.
Let us also point out to a recent work of
Chen et al. \cite{Chenetal2020} who proved multiplier theorems for abstract
self-adjoint operators whose methods do not rely on a-priori heat kernel bounds.
In particular, they obtain a multiplier theorem for $|p|^\alpha+V$ with
$V\geq0$, however again, only in $d=1$ and with $\alpha>1$, see
\cite[Section 5.3]{Chenetal2020} and their Theorem 3.1 and the subsequent
corollary.
In any case, these results are however not applicable in our
situation since we are requiring $\alpha<d$.

Nonetheless, it is possible to establish a spectral multiplier theorem
associated to $\cl_{a,\alpha}$ in two special cases.
On the one hand, a simple computation using the pointwise bounds on
$\me{-\cl_{a,\alpha}}(x,y)$ shows that the heat kernel is bounded on $L^p$
for all $a\geq a_*$, see Lemma \ref{bernstein}.
On the other hand, based on an abstract result by Hebisch \cite{Hebisch1995},
we prove a H\"ormander multiplier theorem for $\cl_{a,\alpha}$ if $a\geq0$.
In fact, we obtain the following result.

\begin{theorem}
  \label{cancellationexa}
  Let $d\in\N$, $\tilde a\geq a>0$, $\alpha\in(0,2\wedge d)$, and 
  $c\in(0,\alpha)$. Moreover, let $V$ be a measurable function on $\R^d$
  satisfying
  \begin{align}
    \label{eq:u}
    \frac{a}{|x|^\alpha} \leq V(x) \leq \frac{\tilde a}{|x|^\alpha}\,.
  \end{align}
  If $F\in H^s_{\mathrm{loc}}(\br)$ with
  $$
  s>2^{[d/(2c)]}\left[\frac{d}{2}\left(1+\frac{1}{c}\right)+1\right]+\frac12
  $$
  and for a $0\neq\varphi\in C_c^\infty(\br_+)$, one has
  $\sup_{t>0}\|\varphi F(t\cdot)\|_{H^s}<\infty$,
  then $F(|p|^\alpha+V)$, initially defined via the $L^2$ functional calculus,
  has weak type $(1,1)$ and is bounded on $L^p(\R^d)$ for all $p\in(1,\infty)$.
\end{theorem}

Immediate consequences of this result are, e.g., the $L^p$-boundedness of
Riesz means $R_\beta(\cl_{a,\alpha})$ where
$R_\beta(\lambda):=(1-\lambda)_+^\beta$ whenever $\beta>s-1/2$, and
imaginary powers $\cl_{a,\alpha}^{i\tau}$, $\tau\in\R$.
In the context of the present work, the main importance of this result is
that it allows us to construct a basic Littlewood--Paley theory by deriving
Bernstein estimates, Lemma \ref{bernstein}, and the crucial square function
estimates, Theorem \ref{squarefunctions}.

\medskip
We would like to emphasize that the above strategy is inspired by
\cite{Killipetal2018,Franketal2019}.
The idea to formulate the norms $\|\cl_{a,\alpha}^{s/2}f\|_p$ in terms of
square functions, which are in turn expressed via the heat kernel, is
borrowed from \cite{Killipetal2018}. The construction of
Littlewood--Paley theory based on heat kernels is, e.g., exhaustively
treated in \cite{Stein1970T}.
On the other hand, we are fortunate to invoke the key estimates on the
Riesz kernel of $\cl_{a,\alpha}$ and on the difference of the heat kernels
of $\cl_{a,\alpha}$ and $|p|^\alpha$ which were obtained in
\cite{Franketal2019}.
In order to make the paper self-contained, we have, however, decided to
review and present the involved arguments for the reader's convenience.

\subsection*{Organization}
In the next section we recall the crucial bounds on the heat kernel of
$\cl_{a,\alpha}$ and state simple but important weighted and ultracontractive
estimates for $\me{-(|p|^\alpha+V)}$ for non-negative functions $V$ on $\R^d$.
These estimates play a major role in the subsequent section where we prove a
H\"ormander multiplier theorem for $|p|^\alpha+V$ with $V$ as in \eqref{eq:u}.
Afterwards, we discuss difficulties arising in the case of negative coupling
constants.
In the fourth section we derive Bernstein estimates and square function
estimates which are crucial to express the $L^p$ norms generated by powers
of $\cl_{a,\alpha}$.
In the fifth section, we prove a reversed Hardy inequality expressed in terms
of square functions and give the proof of Theorem \ref{Thm:A.1}.
In the last section we present a simple generalization of the main result when
the Hardy potential is replaced by a function $V$ which satisfies \eqref{eq:u}.

\section{Heat kernel associated to $\cl_{a,\alpha}$}
We recall recent two-sided bounds on the heat kernel of
$\cl_{a,\alpha}$ by Bogdan et al. \cite{Bogdanetal2019} for $a<0$ and
Cho et al. \cite{Choetal2020} or Jakubowski and Wang \cite{JakubowskiWang2020}
for $a>0$.
For $a=0$ these bounds were already proven by Blumenthal and Getoor
\cite{BlumenthalGetoor1960}. Moreover, for $a=0$ and $\alpha=1$, the
heat kernel is just the Poisson kernel, see also
\cite[Theorem 1.14]{SteinWeiss1971}.

\begin{theorem}[Heat kernels of generalized Hardy operators]
  \label{heatkernel}
  Let $\alpha\in(0,2\wedge d)$, $a\geq a_*$ and $\delta$ be defined by
  \eqref{eq:defdelta}. Then the heat kernel of $\cl_{a,\alpha}$ satisfies
  for all $x,y\in\br^d\setminus\{0\}$ and $t>0$,
  $$
  \me{-t\cl_{a,\alpha}}(x,y)
  \sim \left(1\vee \frac{t^{1/\alpha}}{|x|}\right)^\delta
  \left(1\vee \frac{t^{1/\alpha}}{|y|}\right)^\delta
  t^{-d/\alpha}\left(1\wedge\frac{t^{1+d/\alpha}}{|x-y|^{d+\alpha}}\right).
  $$
\end{theorem}

The following bounds are going to be vital in the proof of the
spectral multiplier theorem for $|p|^\alpha+V$ with $V$ as in \eqref{eq:u}
and follow immediately from Theorem \ref{heatkernel}.
\begin{lemma}
  \label{heatkernelprops}
  Let $\alpha\in(0,2\wedge d)$ and $V$ be a non-negative, measurable
  function on $\R^d$. Then, for all $t>0$ and $c<\alpha$,
  \begin{subequations}
    \begin{align}
      \label{eq:heatkernelpropa}
      & \sup_{y\in\br^d} \int|\me{-t(|p|^\alpha+V)}(x,y)|(1+t^{-1/\alpha}|x-y|)^{c}\, \dx  <\infty \quad \text{and}\\
      \label{eq:heatkernelpropb}
      & \sup_{y\in\br^d} t^{d/\alpha}\int|\me{-t(|p|^\alpha+V)}(x,y)|^2\, \dx  <\infty\,.
    \end{align}
  \end{subequations}
\end{lemma}

\begin{proof}
  By Trotter's formula, it suffices to prove \eqref{eq:heatkernelpropa} and
  \eqref{eq:heatkernelpropb} where the kernel $\me{-t(|p|^\alpha+V)}(x,y)$ is
  replaced by $\me{-t|p|^\alpha}(x,y)$. Moreover, the substitution
  $x\mapsto t^{1/\alpha}x$ shows that it suffices to consider $t=1$.
  Since
  $$
  1\wedge\frac{1}{|x-y|^{d+\alpha}} \sim \frac{1}{(1+|x-y|)^{d+\alpha}}
  $$
  the integral
  $$
  \int_{\br^d}\frac{(1+|x-y|)^{c}}{(1+|x-y|)^{d+\alpha}}\,\dx
  $$
  is finite for all $y\in\R^d$, if $c<\alpha$ which shows
  \eqref{eq:heatkernelpropa}.
  On the other hand, Plancherel's theorem implies
  $$
  \int_{\br^d}|\me{-|p|^\alpha}(x,y)|^2\, \dx
  \sim \int_{\br^d} \me{-2|p|^\alpha}\,\td p = \const
  $$
  which yields the finiteness of the left side of \eqref{eq:heatkernelpropb}.
\end{proof}

\section{A multiplier theorem for $\cl_{a,\alpha}$}

In \cite{Hebisch1995} Hebisch proved a H\"ormander multiplier theorem for
self-adjoint operators if the associated heat kernel satisfies weighted
and ultracontractive estimates and a certain H\"older condition.
The proof is inspired by the one of Zo \cite{Zo1976}, see also
\cite[Section 4-6]{Hebisch1990A}. Although the result holds in
$L^2(M,d\mu)$ where $M$ is some metric space and $\mu$ is some
Borel measure, we will only state it for $M=\R^d$ and $\mu$ being the
Lebesgue measure.

\begin{theorem}[Hebisch {\cite[Theorem 3.1]{Hebisch1995}}]
  \label{hoermander}
  Let $A$ be a non-negative, self-adjoint operator in $L^2(\R^d)$ and
  assume there exist positive numbers $c,b,m$ such that for all $t>0$,
  the bounds
  \begin{align*}
    \sup_{y\in\br^d} \int_{\br^d}|\me{-tA}(x,y)|(1+t^{-1/m}|x-y|)^c\, \dx  <\infty\,,
  \end{align*}
  \begin{align*}
    \sup_{y\in\br^d} t^{d/m} \int_{\br^d}|\me{-tA}(x,y)|^2\, \dx  <\infty\,,
  \end{align*}
  and
  \begin{align}
    \label{eq:heatkernelcancellation}
    \int_{\R^d} |\me{-tA}(x,y)-\me{-tA}(x,z)|\, \dx  \lesssim t^{-b/m}|y-z|^b \quad \text{for all}\ y,z\in\R^d
  \end{align}
  hold. If $F\in H^s_{\mathrm{loc}}(\br)$ with
  $$
  s>2^{[d/(2c)]}\left[\frac{d}{2}\left(1+\frac{1}{c}\right)+1\right]+\frac12
  $$
  and for a $0\neq\varphi\in C_c^\infty(\br_+)$, one has
  $\sup_{t>0}\|\varphi F(t\cdot)\|_{H^s}<\infty$,
  then $F(A)$, initially defined via the $L^2$ functional calculus,
  has weak type $(1,1)$ and is bounded on $L^p$ for all $p\in(1,\infty)$.
\end{theorem}

The rest of this section is devoted to the verification of the assumptions
of this theorem for $A=|p|^\alpha+V$ with $V$ as in \eqref{eq:u},
thereby proving Theorem \ref{cancellationexa}.
Since $V\geq0$, the first two conditions follow immediately from Lemma
\ref{heatkernelprops}.
Verifying the third condition is more delicate since the heat kernel
bounds of Theorem \ref{heatkernel} can only be used after resolving the
absolute value. However, after resolving it, cancellations are not expected
anymore due to the different constants in front of the heat kernel bounds.
Nonetheless, one can first verify the condition for $\me{-|p|^\alpha}$.
Afterwards, using Duhamel's formula and the heat kernel bounds of Theorem
\ref{heatkernel}, we verify the condition also for $\me{-(|p|^\alpha+V)}$.

\begin{proposition}
  \label{cancellationex}
  Let $d\in\bn$, $\alpha\in(0,2)$, and $b\in(0,1]$.
  Then \eqref{eq:heatkernelcancellation} holds for $A=|p|^\alpha$ and
  $m=\alpha$.
\end{proposition}

\begin{proof}
  Translating $x\mapsto x+z$ and scaling $x\mapsto t^{1/\alpha}x$ shows that
  it suffices to prove
  \begin{align}
    \label{eq:cancellationex}
    \int_{\br^d}\left|\me{-|p|^\alpha}(x,0)-\me{-|p|^\alpha}(x,w)\right|\,\dx \lesssim |w|^b
  \end{align}
  where $w=(y-z)/t^{1/\alpha}$.
  Since $\me{-|p|^\alpha}(x)$ is integrable by \eqref{eq:heatkernelpropa},
  it suffices to consider $|w|\leq1/2$.
  We split the integral over $x$ at $|x|=3|w|$ and consider first
  $|x|\leq3|w|$. Since the heat kernel is uniformly bounded in $x$ by
  Theorem \ref{heatkernel}, the triangle inequality yields
  \begin{align*}
    \int\limits_{|x|\leq3|w|}\left|\me{-|p|^\alpha}(x,0)-\me{-|p|^\alpha}(x,w)\right|\,\dx
    \leq 2\int\limits_{|x|\leq4|w|}\left|\me{-|p|^\alpha}(x,0)\right|\,\dx \lesssim |w|^d\,.
  \end{align*}
  For $|x|\geq3|w|$, we use the mean value theorem to estimate the
  left side of \eqref{eq:cancellationex} by a constant times
  $$
  |w| \int\limits_{|x|\geq2|w|}\dx\ \left|\nabla_x\int_{\R^d}\me{ipx}\me{-|p|^\alpha}\,\td p\right|\,.
  $$
  Using the Fourier--Bessel transform (see, e.g., Stein and Weiss
  \cite[Chapter \rom{4}]{SteinWeiss1971}) and the formulas for
  derivatives of Bessel functions \cite[Formula 9.1.30]{Olver1968}, namely
  \begin{align}
    \label{eq:besselrecursion}
    \frac{\td}{\dz}(z^{-\nu} J_\nu(z)) & = -z^{-\nu} J_{\nu+1}(z) \quad \text{for}\ z>0\,,\nu\in\R\,,
  \end{align}
  we obtain for $r=|x|$,
  \begin{align}
    \label{eq:meanvalue}
    \begin{split}
      \left|\nabla \int_{\br^d}\me{ipx}\me{-|p|^\alpha}\,\td p\right|
      & = \left|\partial_r \int_0^\infty k^{d-1}\me{-k^\alpha}(kr)^{-(d-2)/2}J_{(d-2)/2}(kr)\,\dk\right|\\
      & = \left|\int_0^\infty k^{d} \me{-k^\alpha}(kr)^{1-d/2}J_{d/2}(kr)\,\dk\right|\,.
    \end{split}
  \end{align}
  We split the integral over $x$ once more at $|x|=2$ and first show that
  the right side of \eqref{eq:meanvalue} is integrable for $|x|\geq2$.
  To this end, we integrate by parts, using once more
  \eqref{eq:besselrecursion}, and obtain
  \begin{align*}
    & \int_0^\infty k^{d} \me{-k^\alpha}(kr)^{1-d/2}J_{d/2}(kr)\,\dk
    = -r^{-1}\int_0^\infty \me{-k^\alpha}k^d\partial_k\left[(kr)^{1-d/2}J_{d/2-1}(kr)\right]\,\dk\\
    & \quad = r^{-1}\int_0^\infty \me{-k^\alpha}k^{d-1}(d-\alpha k^\alpha)\cdot(kr)^{1-d/2}J_{d/2-1}(kr)\,\dk\,.
  \end{align*}
  The integral over $k$ obviously exists for large $k$ due to the
  $\me{-k^\alpha}$ factor. However, we must be careful with the behavior of
  the integrand for small $k$. Integrating $n-1$ more times by parts shows
  that the right side of the last equation is equal to
    \begin{align}
    \label{eq:meanvalue2}
    \begin{split}
      & r^{1-d/2-n}\int_0^\infty k^{1+d/2-n}J_{d/2-n}(kr)\sum_{j=0}^n a_j\me{-k^\alpha}k^{j\alpha}\,\dk\\
      & \quad = r^{-d}\int_0^\infty (kr)^{1+d/2-n}J_{d/2-n}(kr)\sum_{j=0}^n a_j\me{-k^\alpha}k^{j\alpha}\,\dk
    \end{split}
  \end{align}
  where $a_j=a_j(d,\alpha)\in\R$ and the $k^{j\alpha}$ arise from
  differentiating $\me{-k^\alpha}$. The boundary terms vanish at $k=\infty$
  due to the $\me{-k^\alpha}$ factor. We will momentarily explain why the
  boundary terms also vanish at $k=0$.
  
  We distinguish now between odd and even $d$.
  If $d$ is even, we choose $n=d/2$. Using \eqref{eq:besselrecursion} and
  $J_{-m}(z)=(-1)^mJ_m(z)$ for $m\in\N$ (see \cite[Formula 9.1.5]{Olver1968}),
  the $j$-th summand on the right side of \eqref{eq:meanvalue2} becomes
  \begin{align*}
    & a_j r^{-d}\int_0^\infty (kr)J_0(kr)k^{j\alpha}\me{-k^\alpha}\,\dk
    = -a_j r^{-d-1}\int_0^\infty \partial_k((kr)J_{-1}(kr))k^{j\alpha}\me{-k^\alpha}\,\dk\\
    & \quad = -a_j r^{-d-1}\int_0^\infty k^{-1}\cdot kr J_{1}(kr)\left(\alpha j k^{j\alpha}-\alpha k^{j\alpha+\alpha}\right)\me{-k^\alpha}\,\dk\,,
  \end{align*}
  where the boundary term of the partial integration vanished at $k=0$
  quadratically. Using the bound $|J_1(z)|\lesssim\min\{z,z^{-1/2}\}$
  (see \cite[Formula 9.1.7 and 9.2.1]{Olver1968}), the absolute value of the
  right side of the last formula can be bounded by a constant times
  \begin{align*}
    r^{-d}\left(\int_0^{r^{-1}} kr\cdot (k^{j\alpha}+k^{j\alpha+\alpha})\me{-k^\alpha}\,\dk
    + \int_{r^{-1}}^\infty (kr)^{-1/2} (k^{j\alpha}+k^{j\alpha+\alpha})\me{-k^\alpha}\,\dk\right)\,.
  \end{align*}
  The second summand is bounded by a constant times $r^{-d-1/2}$
  whereas the first summand is bounded by $r^{-d-1-j\alpha}+r^{-d-1-(j+1)\alpha}$.
  Thus, the contribution of even $d$ is integrable for $|x|=r\geq2$ in $\R^d$.
  Note that for $n=d/2-1$, the integrand of \eqref{eq:meanvalue2} is
  \begin{align*}
    (kr)^2 J_1(kr) \sum_{j=0}^n a_j k^{j\alpha}\me{-k^\alpha}
    = -rk^2\partial_k(J_0(kr)) \sum_{j=0}^n a_j k^{j\alpha}\me{-k^\alpha}\,,
  \end{align*}
  i.e., the boundary terms of the partial integration always vanished
  at least quadratically.
    
  If on the other hand $d$ is odd, we choose $n=(d+1)/2$ and use
  \cite[Formula 10.16.1]{NIST:DLMF}, i.e.,
  $J_{-1/2}(kr)= \sqrt{2/\pi}(kr)^{-1/2}\cos(kr)$. Thus, \eqref{eq:meanvalue2}
  becomes
  \begin{align*}
    & \sqrt{\frac{2}{\pi}}r^{-d}\int_0^\infty \me{-k^\alpha}\cos(kr)\sum_{j=0}^n a_j k^{j\alpha}\,\dk\\
    & = \frac{a_0}{\sqrt{2\pi}}r^{-d}\int_\R \me{-|k|^\alpha}\me{ikr}\,\dk
    + \sqrt{\frac{2}{\pi}}r^{-d-1}\int_0^\infty r\cos(kr)\sum_{j=1}^n a_j \me{-k^\alpha}k^{j\alpha}\,\dk\,.
  \end{align*}
  The first integral over $k$ is just the one-dimensional heat kernel
  $\me{-|p|^\alpha}(r)$ which, by Theorem \ref{heatkernel}, decays
  like $r^{-1-\alpha}$. 
  Integrating the second summand once more by parts yields
  \begin{align*}
    -\sqrt{\frac{2}{\pi}}r^{-d-1}\sum_{j=1}^n a_j \int_0^\infty (j\alpha k^{j\alpha-1}-\alpha k^{j\alpha+\alpha-1})\me{-k^\alpha}\sin(kr)\,\dk\,.
  \end{align*}
  This shows that both the integral over $k$, as well as the subsequent integral
  over $\{x\in\R^d:|x|\geq2\}$ exist. Finally, we mention why the boundary
  terms at $k=0$ also vanished in this case.
  If $n=(d-1)/2$, the integrand of \eqref{eq:meanvalue2} is
  $$
  kr\sin(kr) \sum_{j=0}^n a_j k^{j\alpha}\me{-k^\alpha}
  = -(\partial_k\cos(kr))\sum_{j=0}^n a_j k^{j\alpha+1}\me{-k^\alpha}
  $$
  by \cite[Formula 10.16.1]{NIST:DLMF}. This shows that the boundary
  terms vanish at least linearly at $k=0$.\\
  Combining the cases of even and odd $d$ thus shows
  $$
  |w|\int\limits_{|x|\geq2}\left|\nabla\int_{\br^d}\me{ipx}\me{-|p|^\alpha}\,\td p\right|\,\dx
  \lesssim |w|\,.
  $$
  If $2|w|\leq|x|\leq2$, we use \cite[Formula 9.1.60]{Olver1968}, i.e.,
  $|J_{d/2}(kr)|\leq1$, to estimate the right side of \eqref{eq:meanvalue}
  by
  \begin{align*}
    \int_0^\infty k^d\me{-k^\alpha}(kr)^{1-d/2}\,\dk \lesssim r^{1-d/2}\,.
  \end{align*}
  This shows that the integral over $2|w|\leq|x|\leq2$ exists uniformly
  in $|w|$ and thus
  $$
  |w|\int\limits_{2|w|\leq|x|\leq2}\left|\nabla\int_{\br^d}\me{ipx}\me{-|p|^\alpha}\,\td p\right|\,\dx
  \lesssim |w|\,,
  $$
  too.
\end{proof}

We will now use perturbation theory to generalize this result
to $|p|^\alpha+V$ with $V$ as in \eqref{eq:u}.

\begin{proof}[Proof of Theorem \ref{cancellationexa}]
  We merely need to check the H\"older condition
  \eqref{eq:heatkernelcancellation} in Theorem \ref{hoermander} (with
  $m=\alpha$) since the first two conditions were already verified in Lemma
  \ref{heatkernelprops}.
  As in the proof of Proposition \ref{cancellationex}, the fact that for any
  $t>0$,
  \begin{align}
    \label{eq:heatl1}
    \sup_{y\in\R^d}\int_{\R^d}|\me{-t(|p|^\alpha+V)}(x,y)|\,\dx
    \leq \sup_{y\in\R^d}\int_{\R^d}|\me{-t|p|^\alpha}(x,y)|\,\dx \lesssim 1
  \end{align}
  holds, shows that it suffices to prove
  \begin{align}
    \label{eq:cancellationexa}
    \int_{\br^d}\left|\me{-t(|p|^\alpha+V)}(x,w)-\me{-t(|p|^\alpha+V)}(x,y)\right|\,\dx
    \lesssim t^{-b/\alpha}|w-y|^b
  \end{align}
  for some $b>0$ and $t^{-1/\alpha}|w-y|\leq 1/2$.
  By the Duhamel formula
  \begin{align*}
    \me{-t(|p|^\alpha+V)}(x,w) = \me{-t|p|^\alpha}(x,w) - \int_0^t \ds\int_{\R^d}\dz\ \me{-(t-s)(|p|^\alpha+V)}(x,z)V(z)\me{-s|p|^\alpha}(z,w)\,,
  \end{align*}
  the triangle inequality, and the maximum principle, the left side of
  \eqref{eq:cancellationexa} is bounded by
  \begin{align*}
    & \int_{\br^d}\left|\me{-t|p|^\alpha}(x,w)-\me{-t|p|^\alpha}(x,y)\right|\,\dx\\
    & \quad + \tilde a \int_{\br^d}\dx\ \int_0^t \ds\int_{\R^d}\dz\ \me{-(t-s)\cl_{a,\alpha}}(x,z)|z|^{-\alpha}\left|\me{-s|p|^\alpha}(z,w)-\me{-s|p|^\alpha}(z,y)\right|\,.
  \end{align*}
  The relation
  $\me{-t\cl_{a,\alpha}}(x,y)=t^{-d/\alpha}\me{-\cl_{a,\alpha}}(t^{-1/\alpha}x,t^{-1/\alpha}y)$
  (see \cite[Lemma 2.1]{JakubowskiWang2020})
  shows that it suffices to consider $t=1$ again. 
  The assertion for the first summand was already shown in Proposition
  \ref{cancellationex} and any $b\in(0,1]$.
  For $\gamma\in(0,1)$ and $|z|\geq|w-y|^\gamma$, the second summand can be
  estimated using \eqref{eq:heatl1},
  \begin{align*}
    \me{-s|p|^\alpha}(x) \geq \me{-s|p|^\alpha}(y) \quad \text{for all}\ |x|\leq|y|\ \text{and}\ s>0
  \end{align*}
  (see, e.g., \cite[Formula (5.1)]{BlumenthalGetoor1960}), the formula
  for the Riesz kernel of $|p|^\alpha$, i.e.,
  $$
  |p|^{-\alpha}(x,y) = \int_0^\infty \me{-s|p|^\alpha}(x,y)\,\ds
  = \frac{\Gamma((d-\alpha)/2)}{\pi^{d/2}2^\alpha\Gamma(\alpha/2)}|x-y|^{-d+\alpha}
  $$
  (see, e.g., \cite[Chapter \rom{5}, \S 1.1]{Stein1970}), and
  \cite[Theorem 4.5.10]{Hormander1990} (with $1\leq p\leq\infty$ such that
  $0<d-d+\alpha-d/p<1$, i.e., $\alpha-1<d/p<\alpha$,
  and $|z|^{-\alpha}\in L^p(|z|\geq|w-y|^\gamma)$ for $p\in(d/\alpha,\infty]$) by
  \begin{align*}
    & \int_0^1 \ds\ \int_{\br^d}\dx\ \me{-(1-s)\cl_{a,\alpha}}(x,z)\int\limits_{|z|\geq|w-y|^\gamma}\dz\ |z|^{-\alpha} \left|\me{-s|p|^\alpha}(z,w)-\me{-s|p|^\alpha}(z,y)\right|\\
    & \quad \lesssim \int\limits_{|z|\geq|w-y|^\gamma}\dz\ |z|^{-\alpha}\int_0^\infty \ds\ \left(\me{-s|p|^\alpha}(z,w)-\me{-s|p|^\alpha}(z,y)\right)\\
    & \qquad\qquad\qquad\qquad\qquad\qquad \times \left(\one_{\{|z-w|\leq|z-y|\}}-\one_{\{|z-w|\geq|z-y|\}}\right)\\
    & \quad = \frac{\Gamma((d-\alpha)/2)}{\pi^{d/2}2^\alpha\Gamma(\alpha/2)} \int\limits_{|z|\geq|w-y|^\gamma}\dz\ |z|^{-\alpha} \left||z-w|^{-d+\alpha}-|z-y|^{-d+\alpha}\right|\\
    &\quad \lesssim_{d,\alpha}|w-y|^{d-d+\alpha-d/p}\, \||z|^{-\alpha}\|_{L^p(|z|\geq|w-y|^\gamma)}
      = A |w-y|^{\frac{1-\gamma}{p}(\alpha p-d)}\,.
  \end{align*}
  Thus, we are left to examine the case $|z|\leq|w-y|^\gamma$ with the above
  $\gamma<1$. As in the proof of Proposition \ref{cancellationex}, we do not
  expect any further cancellations in this region anymore.
  Therefore, using the triangle inequality, it suffices to estimate the
  contributions from $\me{-s|p|^\alpha}(z,w)$, respectively
  $\me{-s|p|^\alpha}(z,y)$ separately. Without loss of generality, we only
  treat the summand with $\me{-s|p|^\alpha}(z,w)$ and examine closer the
  behavior for $s\lessgtr|w-y|^{\epsilon\alpha}$ and $|z-w|\lessgtr|w-y|^\epsilon$
  with $0<\epsilon<\gamma$.
  On the one hand, for $s\geq|w-y|^{\epsilon\alpha}$ and arbitrary $|z-w|$,
  one estimates (using the maximum principle to perform the integration
  over $x$ and using $\exp(-s|p|^\alpha)(z,w)\lesssim s^{-d/\alpha}$)
  \begin{align*}
    & \int_{|w-y|^{\alpha\epsilon}}^1 \ds \int\limits_{|z|\leq |w-y|^\gamma}\dz\ |z|^{-\alpha}\me{-s|p|^\alpha}(z,w)\\
    & \quad \lesssim |w-y|^{\gamma(d-\alpha)}\int_{|w-y|^{\alpha\epsilon}}^1 \ds\ s^{-d/\alpha}
    \lesssim |w-y|^{(\gamma-\epsilon)(d-\alpha)}\,.
  \end{align*}
  On the other hand, if $|z-w|\geq|w-y|^\epsilon$ and $s\in(0,1)$, one uses
  again the maximum principle to perform the integration over $x$ and
  obtains
  \begin{align*}
    & \int\limits_{\substack{|z|\leq|w-y|^\gamma\\ |z-w|\geq|w-y|^\epsilon}}\dz\ |z|^{-\alpha}|z-w|^{-d-\alpha} \int\limits_0^{|z-w|^\alpha} \ds\ s
    + \int\limits_{\substack{|z|\leq|w-y|^\gamma\\ |z-w|\geq|w-y|^\epsilon}}\dz\ |z|^{-\alpha} \int\limits_{|z-w|^\alpha}^1\ds\ s^{-d/\alpha}\\
    & \quad \lesssim \int\limits_{\substack{|z|\leq|w-y|^\gamma\\ |z-w|\geq|w-y|^\epsilon}}\dz\ |z|^{-\alpha}|z-w|^{-d+\alpha}
      \leq A|w-y|^{(\gamma-\epsilon)(d-\alpha)}
  \end{align*}
  Thus, we are left with the region where $|z-w|\leq |w-y|^\epsilon$ and
  $s\leq|w-y|^{\alpha\epsilon}$ with $\epsilon<\gamma$. At this stage, we
  invoke the heat kernel bounds for $a>0$ of Theorem \ref{heatkernel}.
  Using $1-s\geq1-(1/2)^{\alpha\epsilon}$, $\delta\in(-\alpha,0)$,
  translating $x\mapsto x+z$, and applying H\"older's inequality, we obtain
  \begin{align*}
    & \int \dx\ \int_0^{|w-y|^{\alpha\epsilon}}\ds \int \limits_{\substack{|z|\leq|w-y|^\gamma\\ |z-w|\leq|w-y|^\epsilon}}\dz\ |z|^{-\alpha}\left(1\vee\frac{(1-s)^{1/\alpha}}{|x|}\right)^\delta\left(1\vee\frac{(1-s)^{1/\alpha}}{|z|}\right)^\delta\\
    & \qquad\qquad\qquad\qquad\qquad \times (1-s)^{-d/\alpha}\left(1\wedge\frac{(1-s)^{1+d/\alpha}}{|x-z|^{d+\alpha}}\right)\me{-s|p|^\alpha}(z,w)\\
    & \quad \lesssim \int \frac{\dx}{(1+|x|)^{d+\alpha}} \int\limits_{\substack{|z|\leq|w-y|^\gamma\\ |z-w|\leq|w-y|^\epsilon}}\dz\ |z|^{-\alpha-\delta} \int_0^{|w-y|^{\alpha\epsilon}}\ds\ s^{-d/\alpha}\left(1\wedge\frac{s^{1+d/\alpha}}{|z-w|^{d+\alpha}}\right)\\
    & \quad \lesssim \int\limits_{\substack{|z|\leq|w-y|^\gamma\\ |z-w|\leq|w-y|^\epsilon}}\dz\ |z|^{-\alpha-\delta}|z-w|^{-d+\alpha}\\
    & \quad \leq \left(\int\limits_{|z|\leq|w-y|^\gamma}|z|^{-(\alpha+\delta)p}\,\dz\right)^{1/p}\left(\int\limits_{|z-w|\leq|w-y|^\epsilon}|z-w|^{-(d-\alpha)p'}\,\dz\right)^{1/p'}\\
    & \quad \lesssim |w-y|^{(d-(\alpha+\delta)p)\gamma/p}\cdot |w-y|^{(d-(d-\alpha)p')\epsilon/p'}
      = |w-y|^{\gamma\left(\frac{d}{p}-\alpha-\delta\right)+\epsilon\left(\alpha-\frac{d}{p}\right)}\,.
  \end{align*}
  Both integrals converged since $p<d/(\alpha+\delta)$ and
  $p'<d/(d-\alpha)$, i.e., $p\in(d/\alpha,d/(\alpha+\delta))$ which is
  not an empty interval since $\delta<0$. Moreover, this shows that the
  exponent $\gamma(d/p-\alpha-\delta)+\epsilon(\alpha-d/p)$ is positive
  which concludes the proof of Theorem \ref{cancellationexa}.
\end{proof}

\medskip

Theorem \ref{hoermander} is certainly not applicable if $a<0$ since already
the simple bounds of Lemma \ref{heatkernelprops} do not hold due to the
singularity of $\me{-\cl_{a,\alpha}}(x,y)$ for $|x|,|y|\lesssim1$.
In \cite{Killipetal2018}, Killip et al. proved a Mikhlin multiplier theorem
associated to $-\Delta+a|x|^{-2}$ where they used the fact that the
associated wave equation has the finite speed of propagation property.
This follows from a Paley--Wiener argument and the fact that the heat kernel
satisfies a Davies--Gaffney estimate. In fact, this estimate is also a necessary
condition for the finite speed of propagation property, see, e.g.,
Coulhon and Sikora \cite{CoulhonSikora2008} for further details.
If $\alpha<2$, the distributional support of $\cos(\sqrt{\cl_{a,\alpha}})$ is not
compact anymore which is the main reason why it seems non-trivial to adapt
their proof (cf. \cite[p. 1286f]{Killipetal2018}), even if the coupling constant
is positive.

We conclude with the observation that $\exp(-\cl_{a,\alpha})$ is unbounded on $L^p$
for $a<0$ if $p\notin(d/(d-\delta),d/\delta)$. For this purpose, consider
$\varphi\in C_c^\infty(\br^d)$ with $\supp\phi\subseteq B(0,1)$ such that
$\varphi(x)=1$ for $|x| \leq1/2$. By the lower bound on the heat kernel of
Theorem \ref{heatkernel}, one obtains for $|x|\leq1$,
$$
\left(\me{-\cl_{a,\alpha}}\varphi\right)(x)
\gtrsim |x|^{-\delta}\int\limits_{|y|\leq1/2}\dy\ |y|^{-\delta}\left(1\wedge\frac{1}{|x-y|^{d+\alpha}}\right)
\gtrsim |x|^{-\delta}\,.
$$
Hence, $\me{-\cl_{a,\alpha}}\varphi\notin L^p$ for any $p\geq d/\delta$ and
by self-adjointness and the duality of $L^p$ spaces, it follows that the
$L^p$-boundedness also fails if $p\leq d/(d-\delta)$. This indicates
that $p\in(d/(d-\delta),d/\delta)$ seems to be a ``reasonable'' necessary
condition for a multiplier theorem if $a<0$.


\section{Littlewood--Paley theory}
We define two families of Littlewood--Paley projections associated to
$\cl_{a,\alpha}$ and apply the multiplier theorem to infer their $L^p$-boundedness.
Afterwards, we derive Bernstein estimates and square function estimates.

\begin{definition}[Littlewood--Paley projections]
  Let $\Phi:[0,\infty)\to[0,1]$ be a smooth function such that
  $$
  \Phi(\lambda) = 1 \text{ for } 0\leq\lambda\leq1 \quad\text{and}\quad
  \Phi(\lambda) = 0 \text{ for } \lambda\geq2\,.
  $$
  For each dyadic number $N\in 2^\bz$, let
  $$
  \Phi_N(\lambda) = \Phi(\lambda/N^{\alpha/2}) \quad\text{and}\quad
  \Psi_N(\lambda) = \Phi_N(\lambda)-\Phi_{N/2}(\lambda)
  $$
  such that $\sum_{N\in 2^\Z}\Psi_N(\lambda)=1$ for $\lambda\in\R_+$.
  We define the standard Littlewood--Paley projections as
  $$
  \tilde P_N^{a,\alpha}:= \Psi_N(\sqrt{\cl_{a,\alpha}})
  \quad\text{and}\quad
  \tilde P_N^\alpha:=\tilde P_N^{0,\alpha}
  $$
  and a second set of Littlewood--Paley projections via the heat kernel as
  $$
  P_N^{a,\alpha}:=\me{-\cl_{a,\alpha}/N^\alpha}-\me{-\cl_{a,\alpha}/(N^\alpha/2^\alpha)}
  \quad\text{and}\quad
  P_N^\alpha:=P_N^{0,\alpha}\,.
  $$
\end{definition}

We will now prove Bernstein estimates for these projections which
show in particular that $\me{-\cl_{a,\alpha}}$ is bounded on $L^p$ for
all $a\geq a_*$. In general, these inequalities are useful when the
spectral parameter $\lambda$ is localized, because low Lebesgue
integrability can be upgraded to high Lebesgue integrability at the cost
of some powers of $N$. In fact, this cost is a gain for low $N$ which
improves the inequality.

\begin{lemma}[Bernstein estimates]
  \label{bernstein}
  Let $1<p\leq q<\infty$ when $a\geq0$ and let $d/(d-\delta)<p\leq q<d/\delta$
  when $0>a\geq a_*$. Then the following statements hold.\\
  (1) If $a\geq0$, then
  $\|(\cl_{a,\alpha}/N^{\alpha})^{\frac{s}{2}}\tilde P_N^{a,\alpha}f\|_p\sim \|\tilde P_N^{a,\alpha}f\|_p$,
  i.e.,
  $N^{\frac{\alpha s}{2}}\|\tilde P_N^{a,\alpha}f\|_p\sim \|\cl_{a,\alpha}^{\frac{s}{2}}\tilde P_N^{a,\alpha}f\|_p$
  for all $f\in C_c^\infty(\br^d)$ and all $s\in\br$.\\
  (2) The projections $P_N^{a,\alpha}$ and, if $a\geq0$, $\tilde P_N^{a,\alpha}$
  are bounded from $L^p$ to $L^q$ with norm
  $\mathcal{O}(N^{d\left(\frac1p-\frac1q\right)})$.
\end{lemma}

\begin{proof}
  The first assertion follows immediately from Theorem \ref{cancellationexa}.\\
  We focus now on the second assertion and begin with the observation that
  $\tilde P_N^{a,\alpha}$ can be written as a product of $L^p$-bounded multipliers
  due to Theorem \ref{cancellationexa} and the $L^p\to L^q$-boundedness of
  $\me{-\cl_{a,\alpha}/N^\alpha}$. More precisely, we have for some $r\in(p,q)$
  \begin{align*}
    & \|\tilde P_N^{a,\alpha}f\|_{L^q}\\
    & \quad \leq \|\me{-\cl_{a,\alpha}/N^\alpha}\|_{L^r\to L^q} \|\me{\cl_{a,\alpha}/N^\alpha} \tilde P_N^{a,\alpha}\me{\cl_{a,\alpha}/N^\alpha}\|_{L^r\to L^r} \|\me{-\cl_{a,\alpha}/N^\alpha }\|_{L^p\to L^r}\|f\|_{L^p}\\
    & \quad \lesssim N^{d(1/p-1/r+1/r-1/q)}\|f\|_{L^p}\,.
  \end{align*}
  Thus, it suffices to prove the second assertion for
  $\me{-\cl_{a,\alpha}/N^\alpha}$.
  If $a\geq0$, applying the maximum principle shows that it suffices to
  compute the $L^p\to L^q$-norm of the heat kernel associated to $|p|^\alpha$.
  Scaling $x\mapsto N^{-1}x$ and applying Young's inequality with
  $r=(1+1/q-1/p)^{-1}\geq1$ yields
  \begin{align*}
    \|\me{-\cl_{a,\alpha}/N^\alpha}f\|_q
    \lesssim N^{d} \left\|1\wedge\frac{N^{-\alpha-d}}{|x|^{d+\alpha}}\right\|_r \|f\|_p
    \lesssim N^{d\left(\frac1p-\frac1q\right)} \|f\|_p\,.
  \end{align*}
  For $0>a\geq a_*$, we employ the heat kernel bounds of Theorem
  \ref{heatkernel} to estimate
  \begin{align}
    \label{eq:A.27}
    \begin{split}
      & \|\me{-\cl_{a,\alpha}/N^\alpha}f\|_q\\
      & \quad \lesssim N^{d} \left\|\left(1\vee\frac{N^{-1}}{|x|}\right)^\delta\int_{\br^d}\left(1\vee\frac{N^{-1}}{|y|}\right)^\delta\left(1\wedge\frac{N^{-\alpha-d}}{|x-y|^{d+\alpha}}\right) |f(y)| \,\dy \right\|_q\,.
    \end{split}
  \end{align}
  To handle the right side, we distinguish between the following
  four cases.
  \subsection*{Case 1: $|x| \leq N^{-1}$, $|y| \leq N^{-1}$.}
  Using H\"older's inequality and recalling $d/(d-\delta)<p\leq q<d/\delta$,
  one can estimate the right side of \eqref{eq:A.27} by
  \begin{subequations}
    \label{eq:A.19}
    \begin{align}
      \begin{split}
        & N^{d-2\delta}\left\| |x|^{-\delta}\int\limits_{|y|\leq N^{-1}}|y|^{-\delta}|f(y)|\,\dy\right\|_{L^q(|x| \leq N^{-1})}\\
        & \quad \lesssim N^{d-2\delta}\left\| |x|^{-\delta}\right\|_{L^q(|x| \leq N^{-1})}\left\| |y|^{-\delta}\right\|_{L^{p'}(|y| \leq N^{-1})} \|f\|_p
        \lesssim N^{d\left(\frac1p-\frac1q\right)} \|f\|_p.
      \end{split}
    \end{align}
    
    \subsection*{Case 2: $|x| \leq N^{-1}$, $|y|>N^{-1}$.}
    Using H\"older's inequality, the right side of \eqref{eq:A.27} can be
    estimated by
    \begin{align}
      \begin{split}
        &N^{d-\delta}\left\| |x|^{-\delta}\int_{\br^d}\left(1\wedge\frac{N^{-\alpha-d}}{|x-y|^{d+\alpha}}\right)|f(y)|\,\dy\right\|_{L^q(|x| \leq N^{-1})}\\
        \lesssim& N^{d-\delta} \| |x|^{-\delta}\|_{L^q(|x| \leq N^{-1})} \left\|\left(1\wedge\frac{N^{-\alpha-d}}{|y|^{d+\alpha}}\right)\right\|_{p'} \|f\|_p
        \lesssim N^{d\left(\frac1p-\frac1q\right)} \|f\|_p.
      \end{split}
    \end{align}
    
    \subsection*{Case 3: $|x|>N^{-1}$, $|y| \leq N^{-1}$.}
    Using Minkowski's inequality and then H\"older's inequality, the right
    side of \eqref{eq:A.27} can be bounded by
    \begin{align}
      \begin{split}
        &N^{d-\delta} \left\|\int\limits_{|y| \leq N^{-1}} |y|^{-\delta}\left(1\wedge\frac{N^{-\alpha-d}}{|x-y|^{d+\alpha}}\right)|f(y)|\,\dy\right\|_q\\
        \lesssim&N^{d-\delta} \left\|\left(1\wedge\frac{N^{-\alpha-d}}{|x|^{d+\alpha}}\right)\right\|_q \| |y|^{-\delta}\|_{L^{p'}(|y| \leq N^{-1})} \|f\|_p
        \lesssim N^{d\left(\frac1p-\frac1q\right)} \|f\|_p.
      \end{split}
    \end{align}
    
    \subsection*{Case 4: $|x|>N^{-1}$, $|y|>N^{-1}$.}
    As in the case of non-negative couplings, we employ Young's inequality to
    estimate the last contribution of the right side of \eqref{eq:A.27} by
    \begin{align}
      \begin{split}
        &N^{d} \left\|\int_{\br^d}\left(1\wedge\frac{N^{-\alpha-d}}{|x-y|^{d+\alpha}}\right)|f(y)|\,\dy\right\|_q\\
        \lesssim&N^{d} \left\|\left(1\wedge\frac{N^{-\alpha-d}}{|x|^{d+\alpha}}\right)\right\|_r \|f\|_p
        \lesssim N^{d\left(\frac1p-\frac1q\right)} \|f\|_p
      \end{split}
    \end{align}
  \end{subequations}
  where $1+1/q=1/r+1/p$.
\end{proof}

The spectral multiplier theorem, a randomization argument involving
Khintchine's inequality, and a duality argument yield the following
two-sided square function estimates.
Since the same arguments already appear in the proof of
\cite[Theorem 4.3]{Killipetal2016}, we skip the proof. 

\begin{theorem}[Square function estimates]
  \label{squarefunctions}
  Let $\alpha\in(0,2\wedge d)$, $a\geq0$, $1<p<\infty$, and $s>0$.
  If $k\in\N$ satisfies $k>s/2$, then
  $$
  \left\|\left(\sum_{N\in2^\bz} |N^{\frac{\alpha s}{2}}\tilde P_N^{a,\alpha}f|^2 \right)^{1/2}\right\|_{p}
  \sim \left\|\cl_{a,\alpha}^{s/2}f\right\|_{p}
  \sim \left\|\left(\sum_{N\in2^\bz} |N^{\frac{\alpha s}{2}}(P_N^{a,\alpha})^kf|^2\right)^{1/2}\right\|_{p}
  $$
  for all $f\in C_c^\infty(\br^d)$.
\end{theorem}

Although we apply these estimates only for $k=1$ (since $s\in(0,2)$),
we remark that $k>s/2$ guarantees that
$(N^{\alpha s/2}\lambda^{-s/2})\cdot(\me{-\lambda/N^\alpha}-\me{-\lambda/(N/2)^\alpha})^k$
is actually a H\"ormander multiplier. 

We conclude this section by noting that the spectral theorem in $L^2$
yields an expansion of the identity of $L^2$ functions in terms of
eigenfunctions of $\cl_{a,\alpha}$ as in \cite{Killipetal2016,Killipetal2018}.
More precisely, one has the $L^2$-convergence (noting that zero is
not an eigenvalue of $\cl_{a,\alpha}$)
$$
\sum_{N\in2^{\Z}}(\tilde P_N^{a,\alpha}f)(x)
=f(x)
=\sum_{N\in2^{\Z}}(P_N^{a,\alpha}f)(x)
$$
for all $a\geq a_*$ and $f\in L^2(\R^d)$.
By the spectral multiplier theorem, respectively the $L^p$-boundedness
of the heat kernel, the first equality continues to hold in $L^p$ for
$a\geq0$ and $1<p<\infty$ and the second one for any $1<p<\infty$ if
$a\geq0$ and for any $d/(d-\delta)<p<d/\delta$ if $a<0$.
This is because partial sums are $L^p$-bounded for the above $p$
which allows one to conclude via a density and interpolation (via
H\"older's inequality) argument.


\section{A reverse Hardy inequality and proof of Theorem \ref{Thm:A.1}}

The key tool to prove the reversed Hardy inequality is a pointwise bound
on the difference of the heat kernels of $\cl_{a,\alpha}$ and
$|p|^\alpha$, i.e.,
$$
K_t^\alpha(x,y):= \me{-t |p|^\alpha}(x,y) - \me{-t\cl_{a,\alpha}}(x,y) \,.  
$$
In \cite[Lemma 3.1]{Franketal2019} it was shown that there is an
effective cancellation in the region $(|x|\vee |y|)^\alpha\geq t$
and $|x|\sim |y|$.
There, the bound was formulated in terms of the functions
$$
L_t^{\alpha,\delta}(x,y) := \one_{\{(|x|\vee|y|)^\alpha\leq t\}} t^{-\frac{d}{\alpha}} \left( \frac{t^{2/\alpha}}{|x||y|} \right)^{\delta}
+ \one_{\{ (|x|\vee|y|)^\alpha\geq t\}} \frac{t}{(|x|\vee|y|)^{d+\alpha}} \left( 1 \vee \frac{t^{1/\alpha}}{|x|\wedge|y|} \right)^{\delta}
$$
and
$$
M_t^\alpha(x,y) := \one_{\{(|x|\vee|y|)^\alpha\geq t\}} \one_{\{\frac12 |x|\leq |y|\leq 2|x|\}}
\frac{t^{1-\frac{d}{\alpha}}}{(|x|\wedge|y|)^\alpha} \left( 1 \wedge \frac{t^{1+\frac{d}{\alpha}}}{|x-y|^{d+\alpha}} \right).
$$
Recall that $\delta_+=0$ if $a\geq0$ and $\delta_+=\delta$ if $a<0$.

\begin{lemma}[Difference of kernels]
  \label{differencekernel}
  Let $\alpha\in(0,2\wedge d)$, $a\in[a_*,\infty)$ and $\delta$ be defined
  by \eqref{eq:defdelta}. Then for all $x,y\in\br^d\setminus\{0\}$ and $t>0$,
  \begin{align}
    \label{eq:4.7}
    |K_t^\alpha(x,y)| \lesssim L_t^{\alpha,\delta_+}(x,y) + M_t^\alpha(x,y) \,.
  \end{align}
\end{lemma}

Using this lemma, we formulate and prove a reversed Hardy
inequality for the difference $\cl_{a,\alpha}^{s/2}-|p|^{\alpha s/2}$,
expressed in terms of square functions.
\begin{proposition}[Reverse Hardy inequality in $L^p$]
  \label{reversehardylp}
  Let $s\in(0,2)$, $\alpha\in(0,2\wedge d)$, $a\geq a_*$,
  $\delta$ be defined by \eqref{eq:defdelta},
  $p\in(1,\infty)$ if $a\geq0$,
  and $p\in(d/(d-\delta),d/\delta)$ if $a<0$.
  Then,
  \begin{align}
    \label{eq:reversehardylp}
    \left\|\left(\sum_{N\in 2^\Z}|N^{\alpha s/2}P_N^\alpha f|^2\right)^{1/2}-\left(\sum_{N\in 2^\bz}|N^{\alpha s/2}P_N^{a,\alpha} f|^2\right)^{1/2}\right\|_p
    \lesssim_{d,\alpha,a,s} \| |x|^{-\alpha s/2}f\|_p
  \end{align}
  for all $f\in C_c^\infty(\br^d)$.
\end{proposition}

\begin{proof}
  By the triangle inequality in $\ell^2$, the
  $\ell^1\hookrightarrow \ell^2$-embedding, and Lemma \ref{differencekernel},
  we estimate
  \begin{align}
    \begin{split}
      \label{eq:schur}
      & \left\|\left(\sum_{N\in 2^\bz}|N^{\alpha s/2}P_N^\alpha f|^2\right)^{1/2}-\left(\sum_{N\in 2^\bz}|N^{\alpha s/2}P_N^{a,\alpha} f|^2\right)^{1/2}\right\|_p\\
      & \quad \leq \left\|\left(\sum_{N\in 2^\bz}|N^{\frac{\alpha s}{2}}(P_N^\alpha-P_N^{a,\alpha})f|^2\right)^{\frac12}\right\|_p
      \lesssim \left\|\int_{\R^d} \dy\ \sum_{N\in 2^{\bz}}N^{\frac{\alpha s}{2}} |K_{N^{-\alpha}}^\alpha(x,y)|\, |f(y)|\right\|_p \\
      & \quad \leq \left\|\int_{\R^d} \dy\ \left(\sum_{N\in 2^{\bz}}N^{\frac{\alpha s}{2}} |L_{N^{-\alpha}}^{\alpha,\delta_+}(x,y)| |y|^{\alpha\frac{s}{2}}\right)\frac{|f(y)|}{|y|^{\alpha\frac{s}{2}}}\right\|_p\\
      & \qquad + \left\|\int_{\R^d} \dy\ \left(\sum_{N\in 2^{\bz}}N^{\frac{\alpha s}{2}} |M_{N^{-\alpha}}^\alpha(x,y)| |y|^{\alpha\frac{s}{2}}\right)\frac{|f(y)|}{|y|^{\alpha\frac{s}{2}}}\right\|_p\,.
    \end{split}
  \end{align}
  Thus, it suffices to show that the right side is bounded by
  $\| |x|^{-\alpha s/2} f\|_p$ for all $f\in C_c^\infty(\br^d)$.
  To simplify notation, let $g(x):=|x|^{-\alpha s/2}|f(x)|$.\\
  As in the proof of \cite[Proposition 1.5]{Franketal2019}, we use Schur tests
  to prove the assertion.
  We begin by estimating the first summand and obtain
  \begin{align*}
    \sum_{N\in 2^\bz} N^{\frac{\alpha s}{2}}L_{N^{-\alpha}}^{\alpha,\delta_+}(x,y)
    & = (|x||y|)^{-\delta_+} \sum_{N\leq(|x| \vee |y|)^{-1}} N^{\frac{\alpha s}{2} + d-2\delta_+}\\
    & \qquad + \sum_{N\geq(|x| \vee |y|)^{-1}} N^{\frac{\alpha s}{2}}\, \frac{N^{-\alpha}}{(|x|\vee|y|)^{d+\alpha}} \left( 1 \vee \frac{N^{-1}}{|x|\wedge|y|} \right)^{\delta_+} \\
    & \sim (|x||y|)^{-\delta_+}(|x|\vee|y|)^{-\frac{\alpha s}{2}-d+2\delta_+}
      +\frac{1}{(|x|\vee|y|)^{\frac{\alpha s}{2}+d}}\left(\frac{|x|\vee|y|}{|x|\wedge|y|} \right)^{\delta_+}
  \end{align*}
  where the summability relied on $s<2$ and $\alpha s/2 + d-2\delta_+>0$
  which in turn follows from $\delta\leq(d-\alpha)/2$.
  Observe that the two summands are equal since
  $$
  \frac{|x|\vee|y|}{|x||y|}=\frac{1}{|x|\wedge|y|}\,.
  $$
  Thus,
  \begin{align*}
    \left\|\int_{\br^d} \dy\ \sum_{N\in 2^\bz} N^{\frac{\alpha s}{2}}L_{N^{-\alpha}}^{\alpha,\delta_+}(x,y) |y|^{\alpha\frac{s}{2}}g(y)\right\|_p
    \lesssim \left\|\int_{\br^d} \dy\ \frac{(|x|\vee|y|)^{2\delta_+-d}}{(|x| |y|)^{\delta_+}}\,g(y)\right\|_p \,.
  \end{align*}
  For any $(p\vee p')\delta_+<\beta<(p\wedge p')(d-\delta_+)$ (such
  $\beta$ exist since $d-2\delta\geq\alpha>0$ and
  $d/(d-\delta_+)<p,p'<d/\delta_+$), we have
  $$
  \sup_{y\in\br^d} \int_{\br^d} \dx\, \left( \frac{|y|}{|x|} \right)^{\frac{\beta}{p}}\frac{(|x|\vee|y|)^{2\delta_+-d}}{(|x||y|)^{\delta_+}}
  =\int_{\br^d} \frac{(1\vee|z|)^{2\delta_+-d}}{|z|^{\delta_++\frac{\beta}{p}}}\,\dz<\infty
  $$
  and similarly, since the integral kernel is symmetric in $x$ and $y$,
  $$
  \sup_{x\in\br^d} \int_{\br^d} \dy\, \left( \frac{|x|}{|y|} \right)^{\frac{\beta}{p'}} \frac{(|x|\vee|y|)^{2\delta_+-d}}{(|x||y|)^{\delta_+}}
  =\int_{\br^d} \frac{(1\vee|z|)^{2\delta_+ - d}}{|z|^{\delta_+ +\frac{\beta}{p'}}}\,\dz<\infty \,.
  $$
  Thus, by a weighted Schur test,
  $$
  \left\|\int_{\br^d} \dy\ \frac{(|x|\vee|y|)^{2\delta_+ - d}}{(|x||y|)^{\delta_+}} g(y)\right\|_p
  \lesssim \|g\|_p
  $$
  which shows that the first term in \eqref{eq:schur} satisfies the claimed
  bound.

  Since $|x|\sim|y|$ on the support of $M_{N^{-\alpha}}^\alpha$, the second
  term in \eqref{eq:schur} is bounded by a constant times
  \begin{align*}
    \norm{\int_{\R^d} \dy \sum_{N\in2^\Z} N^{\frac{\alpha s}{2}} \, M_{N^{-\alpha}}^\alpha(x,y) (|x||y|)^\frac{\alpha s}{4} g(y) }_p\,.
  \end{align*}
  Since the integral kernel is symmetric in $x$ and $y$,
  the $L^p$-boundedness will follow from a single Schur test.
  We first estimate
  \begin{align*}
    & \sup_{y\in\R^d} \int_{\R^d} \dx \sum_{N\in2^\Z} N^{\frac{\alpha s}{2}}\, M_{N^{-\alpha}}^\alpha(x,y)  (|x||y|)^\frac{\alpha s}{4} \\
    & = \sup_{y\in\R^d} \int\limits_{\frac12 |y|\leq |x|\leq 2|y|} \dx \sum_{N\geq(|x|\vee|y|)^{-1}} N^{\frac{\alpha s}{2}}\, \frac{N^{-\alpha+d}}{(|x|\wedge |y|)^\alpha} \left( 1 \wedge \frac{N^{-\alpha-d}}{|x-y|^{d+\alpha}} \right)  (|x||y|)^\frac{\alpha s}{4} \\
    & \lesssim \sup_{y\in\R^d} |y|^{\frac{\alpha s}{2}-\alpha} \int\limits_{\frac12 |y|\leq |x|\leq 2|y|} \dx \sum_{N\geq(2|y|)^{-1}} N^{\frac{\alpha s}{2} -\alpha +d} \, \left( 1 \wedge \frac{N^{-\alpha-d}}{|x-y|^{d+\alpha}} \right)  \,.
  \end{align*}
  Interchanging the order of integration and summation shows that
  the right side is bounded by
  \begin{align*}
    & \sup_{y\in\R^d} |y|^{\frac{\alpha s}{2}-\alpha} \sum_{N\geq(2|y|)^{-1}} N^{\frac{\alpha s}{2} -\alpha +d} \int\limits_{\frac12 |y|\leq |x|\leq 2|y|} \dx \left(1\wedge\frac{N^{-\alpha-d}}{\abs{x-y}^{d+\alpha}}\right)\\
    & \quad \leq \sup_{y\in\R^d} |y|^{\frac{\alpha s}{2}-\alpha} \sum_{N\geq(2|y|)^{-1}} N^{\frac{\alpha s}{2} -\alpha +d} \int_{\R^d} \dx \left(1\wedge\frac{N^{-\alpha-d}}{\abs{x-y}^{d+\alpha}}\right)\\
    & \quad \sim \sup_{y\in\R^d} |y|^{\frac{\alpha s}{2}-\alpha} \sum_{N\geq(2|y|)^{-1}} N^{\frac{\alpha s}{2} -\alpha} \sim 1\,.
  \end{align*}
  Thus, the Schur test implies
  \begin{align*}
    \norm{\int_{\R^d} \dy \sum_{N\in2^\Z} N^{\frac{\alpha s}{2}} \, M_{N^{-\alpha}}^\alpha(x,y) (|x||y|)^\frac{\alpha s}{4} g(y) }_p \lesssim \|g\|_p
  \end{align*}
  which shows the asserted inequality.
\end{proof}

We now show that Theorem \ref{Thm:A.1} is an immediate consequence of
Propositions \ref{genhardy} and \ref{reversehardylp} and the
Littlewood--Paley theory from the last section.

\begin{proof}[Proof of Theorem \ref{Thm:A.1}]
  In the following, we always assume $1<p<\infty$.
  If $s\in(0,2)$ and $a\geq0$ (i.e., $\delta\leq0$), the assertion
  $$
  \| |p|^{\alpha s/2}f\|_p \lesssim_{d,p,\alpha,s} \|\cl_{a,\alpha}^{s/2}f\|_p
  $$
  follows from Theorem \ref{squarefunctions}, the triangle inequality,
  Proposition \ref{genhardy} (requiring $\alpha s/2+\delta<d/p$),
  and Proposition \ref{reversehardylp}.
  More precisely,
  \begin{align*}
    \| |p|^{\alpha s/2} f\|_p
    & \sim \left\|\left(\sum_{N\in 2^\bz}|N^{\alpha s/2}P_N^\alpha f|^2\right)^{1/2}\right\|_p \\
    & \leq \left\|\left(\sum_{N\in 2^\bz}|N^{\alpha s/2}P_N^{a,\alpha} f|^2\right)^{1/2}\right\|_p \\
    & \quad +\left\|\left(\sum_{N\in 2^\bz}|N^{\alpha s/2}P_N^\alpha f|^2\right)^{1/2}-\left(\sum_{N\in 2^\bz}|N^{\alpha s/2}P_N^{a,\alpha} f|^2\right)^{1/2}\right\|_p \\
    & \lesssim \left\|\cl_{a,\alpha}^{s/2}f\right\|_p+\left\||x|^{-\alpha s/2}f\right\|_p
      \lesssim \left\|\cl_{a,\alpha}^{s/2}f\right\|_p\,.
  \end{align*}
  Note that the condition $\alpha s/2<d$ in Proposition \ref{genhardy}
  is automatically satisfied since we assumed $s\leq2$ and $\alpha<d$.
  
  The other inequality, i.e.,
  $$
  \|\cl_{a,\alpha}^{s/2}f\|_p \lesssim_{d,p,\alpha,s} \| |p|^{\alpha s/2}f\|_p
  $$
  is proven analogously, but employs \eqref{eq:standardhardy} (with $\alpha s$
  instead of $\alpha$ and requiring $\alpha s/2<d/p$) instead of
  Proposition \ref{genhardy}.

  If $s=2$ and $a\geq a_*$, the inequality
  $\|\cl_{a,\alpha}f\|_p\lesssim \| |p|^\alpha f\|_p$
  follows from the triangle inequality and the ordinary Hardy inequality
  \eqref{eq:standardhardy} (with $2\alpha$ instead of $\alpha$ and
  requiring $\alpha<d/p$).
  The other inequality follows from
  $$
  \| |p|^\alpha f\|_p\leq \|(|p|^\alpha-\cl_{a,\alpha})f\|_p + \|\cl_{a,\alpha} f\|_p
  = \| |x|^{-\alpha}f\|_p+ \|\cl_{a,\alpha} f\|_p
  $$
  and the generalized Hardy inequality, Proposition \ref{genhardy}
  (requiring $\alpha+\delta<d/p<d-\delta$).
\end{proof}


\section{Non-power-like potentials}
As in \cite{Franketal2019}, it is possible to generalize Theorem
\ref{Thm:A.1} to the operator $|p|^\alpha+V$ where $V$ is a function on
$\br^d$ satisfying
\begin{align}
  \label{eq:u2}
  \frac{a}{|x|^\alpha} \leq V(x) \leq \frac{\tilde a}{|x|^\alpha}
\end{align}
with $a_*\leq a\leq \tilde a<\infty$.
We prove the following result.

\begin{theorem}
  \label{Thm:A.1gen}
  Let $\alpha\in(0,2\wedge d)$ and $s\in(0,2]$.
  Let $a_*\leq a\leq \tilde a<\infty$ if $s=2$ and
  $0<a\leq \tilde a<\infty$ if $s\in(0,2)$.
  Let furthermore $\delta=\delta(a)$ be defined by \eqref{eq:defdelta}.
  \begin{enumerate}
  \item[(1)] If $1<p<\infty$ satisfies
    $\alpha s/2+\delta<d/p<\min\{d,d-\delta\}$,
    then for any $V$ satisfying \eqref{eq:u2},
    \begin{align}
      \label{eq:12agen}
      \| |p|^{\alpha s/2}f\|_{L^p(\R^d)} \lesssim_{d,\alpha,a,s} \|(|p|^\alpha+V)^{s/2}f\|_{L^p(\R^d)}
      \quad \text{for all}\ f\in C_c^\infty(\R^d) \,.
    \end{align}
  \item[(2)]
    If $\alpha s/2<d/p<d$ (which already ensures $1<p<\infty$),
    then for any $V$ satisfying \eqref{eq:u2},
    \begin{align}
    \label{eq:12bgen}
      \|(|p|^\alpha+V)^{s/2}f\|_{L^p(\R^d)} \lesssim_{d,\alpha,a,s} \| |p|^{\alpha s/2}f\|_{L^p(\R^d)}
      \quad\text{for all}\ f\in C_c^\infty(\R^d) \,.
    \end{align}
  \end{enumerate}
\end{theorem}

The proof of this theorem is akin to the one of Theorem \ref{Thm:A.1}.
If $s=2$, we merely use the triangle inequality, the ordinary Hardy
inequality, and a modification of the generalized Hardy inequality,
Proposition \ref{genhardy}.
If $s\in(0,2)$, we apply the Littlewood--Paley theory of Section 4,
i.e., square function estimates adapted to $|p|^\alpha+V$, and a
modification of the reversed Hardy inequality, Proposition
\ref{reversehardylp}. These modifications are summarized in
the following propositions whose proofs are analogous to those
in \cite{Franketal2019}, respectively Theorem \ref{squarefunctions},
i.e., \cite[Theorem 4.3]{Killipetal2016}.
We begin with the modified, generalized Hardy inequality.

\begin{proposition}
  Let $1<p<\infty$, $\alpha\in(0,2\wedge d)$,
  $a_*\leq a\leq \tilde a<\infty$, $\delta=\delta(a)$ be defined by
  \eqref{eq:defdelta}, and $\alpha s/2\in(0,d)$.
  If $s$ and $p$ satisfy $\alpha s/2+\delta<d/p<d-\delta$,
  then for any $V$ satisfying \eqref{eq:u2},
  \begin{align*}
    \| |x|^{-\alpha s/2}f\|_p \lesssim_{d,\alpha,a,s} \|(|p|^\alpha+V)^{s/2}f\|_p
    \quad\text{for all}\ f\in C_c^\infty(\br^d) \,.
  \end{align*}
\end{proposition}

\begin{proof}
  By Trotter's formula, we have for all $x,y\in\br^d$ and $t>0$,
  \begin{align*}
    0\leq e^{-t(|p|^\alpha+V)}(x,y) \leq e^{-t\cl_{a,\alpha}}(x,y) \,.
  \end{align*}
  By the spectral theorem, i.e.,
  $$
  \cl_{a,\alpha}^{-s/2}(x,y)=\frac{1}{\Gamma(s/2)}\int_0^\infty\me{-t\cl_{a,\alpha}}(x,y)t^{s/2}\frac{\dt}{t}\,,
  $$
  and analogously for $(|p|^\alpha+V)^{-s/2}(x,y)$, it follows that
  $$
  (|p|^\alpha+V)^{-s/2}(x,y) \leq \cl_{a,\alpha}^{-s/2}(x,y) \,.
  $$
  Therefore, the upper bound \eqref{eq:rieszhardy} on $\cl_{a,\alpha}^{-\frac s2}(x,y)$
  continues to hold for $(|p|^\alpha +V)^{-\frac s2}(x,y)$, so the claim follows
  as in Proposition \ref{genhardy}. 
\end{proof}

As in Section 4, we define Littlewood--Paley projections associated
to $|p|^\alpha+V$ via the heat kernel as
$$
P_N^{V,\alpha}:=\me{-(|p|^\alpha+V)/N^\alpha}-\me{-(|p|^\alpha+V)/(N^\alpha/2^\alpha)}\,.
$$
By Theorem \ref{cancellationexa} and the arguments of
\cite[Theorem 4.3]{Killipetal2016}, we obtain the following square
function estimates.

\begin{proposition}
  \label{squarefunctionsgen}
  Let $\alpha\in(0,2\wedge d)$, $\tilde a\geq a>0$, $1<p<\infty$, and $s>0$.
  If $k\in\N$ satisfies $k>s/2$, then for any $V$ satisfying \eqref{eq:u2}, 
  $$
  \left\|(|p|^\alpha+V)^{s/2}f\right\|_{p}
  \sim \left\|\left(\sum_{N\in2^\bz} |N^{\frac{\alpha s}{2}}(P_N^{V,\alpha})^kf|^2\right)^{1/2}\right\|_{p}
  $$
  for all $f\in C_c^\infty(\br^d)$.
\end{proposition}

Finally, we have the following modified reversed Hardy inequality.
\begin{proposition}
  Let $\alpha\in(0,2\wedge d)$, $a_*\leq a\leq \tilde a<\infty$,
  $\delta=\delta(a)$ be defined by \eqref{eq:defdelta}, $p\in(1,\infty)$
  if $a\geq0$, and $p\in(d/(d-\delta),d/\delta)$ if $a<0$.
  Let furthermore $s\in(0,2)$.
  Then, for any $V$ satisfying \eqref{eq:u2},
  $$
  \left\|\left(\sum_{N\in 2^\bz}|N^{\alpha s/2}P_N^\alpha f|^2\right)^{1/2}-\left(\sum_{N\in 2^\bz}|N^{\alpha s/2}P_N^{V,\alpha} f|^2\right)^{1/2}\right\|_p
  \lesssim_{d,\alpha,a,s} \| |x|^{-\alpha s/2}f\|_p
  $$
  for all $f\in C_c^\infty(\br^d)$.
\end{proposition}    

\begin{proof}
  Denoting
  $$
  \tilde K_t^\alpha(x,y) := e^{-t|p|^\alpha}(x,y) - e^{-t(|p|^\alpha +V)}(x,y)\,,
  $$
  Trotter's formula yields for any $x,y\in\R^d$ and $t>0$,
  $$
  e^{-t|p|^\alpha}(x,y) - e^{-t\cl_{a,\alpha}}(x,y) \leq 
  \tilde K_t^\alpha(x,y) \leq e^{-t|p|^\alpha}(x,y) - e^{-t\cl_{\tilde a,\alpha}}(x,y) \,.
  $$
  Since $\tilde \delta:=\Psi_{\alpha,d}^{-1}(\tilde a)\leq\delta$,
  Lemma \ref{differencekernel} with $a$ and $\tilde a$ implies
  $$
  |\tilde K_t^\alpha(x,y)| \lesssim L_t^{\alpha,\tilde\delta_+}(x,y) + L_t^{\alpha,\delta_+}(x,y) + M_t^\alpha(x,y)
  \lesssim L_t^{\alpha,\delta_+}(x,y) + M_t^\alpha(x,y)\,.
  $$
  Using this estimate, the assertion follows in the same way as
  Proposition \ref{reversehardylp}.
\end{proof}



\def\cprime{$'$}

\end{document}